\documentclass[reqno]{amsart}
\usepackage{amssymb}
\usepackage[mathscr]{eucal}
\usepackage{hyperref}
\usepackage[all]{xy}

\theoremstyle{plain}
\newtheorem{prop}{Proposition}[section]
\newtheorem{thm}[prop]{Theorem}
\newtheorem{cor}[prop]{Corollary}
\newtheorem{lem}[prop]{Lemma}

\theoremstyle{definition}
\newtheorem{dfn}[prop]{Definition}
\newtheorem{rem}[prop]{Remark}
\newtheorem{rems}[prop]{Remarks}
\newtheorem{notat}[prop]{Notation}
\newtheorem{lab}[prop]{}

\newcommand{\isoto}{\overset{\sim}{\to}}
\newcommand{\labelto}[1]{\overset{#1}{\longrightarrow}}
\renewcommand{\iff}{\Leftrightarrow}
\renewcommand{\subset}{\subseteq}

\newcommand{\C}{{\mathbb{C}}}
\newcommand{\G}{{\mathbb{G}}}
\newcommand{\N}{{\mathbb{N}}}
\renewcommand{\P}{{\mathbb{P}}}
\newcommand{\R}{{\mathbb{R}}}
\newcommand{\Z}{{\mathbb{Z}}}

\newcommand{\scrE}{{\mathscr{E}}}
\newcommand{\scrK}{{\mathscr{K}}}
\newcommand{\scrO}{{\mathscr{O}}}

\DeclareMathOperator{\Div}{Div}
\renewcommand{\div}{{\rm div}}
\DeclareMathOperator{\Gal}{Gal}

\DeclareMathOperator{\Pic}{Pic}

\DeclareMathOperator{\rk}{rk}
\DeclareMathOperator{\Spec}{Spec}
\DeclareMathOperator{\spn}{span}

\DeclareTextFontCommand{\textnf}{\normalfont}
\newcommand{\et}{{\textnf{\'et}}}

\newcommand{\reg}{{\rm reg}}
\newcommand{\sing}{{\rm sing}}

\newcommand{\x}{{\mathtt{x}}}

\renewcommand{\setminus}{\smallsetminus}
\renewcommand{\epsilon}{\varepsilon}
\newcommand{\mal}{\,.\,}
\newcommand{\ol}{\overline}
\newcommand{\plus}{{\scriptscriptstyle+}}
\newcommand{\wh}[1]{\widehat{#1}}
\newcommand{\To}{\Rightarrow}

\newcommand{\Label}[1]{\label{#1}}

\begin{document}

\title
[Extreme positive ternary sextics]
{Extreme positive ternary sextics}

\author{Aaron Kunert, \ Claus Scheiderer}

\subjclass[2010]
{Primary 14\,P\,05;
secondary
14\,C\,22,
14\,H\,45}

\address{Fachbereich Mathematik und Statistik, Universit\"at Konstanz,
  Germany}

\maketitle

%-------------------------------------------------------------------%

\begin{abstract}
We study nonnegative (psd) real sextic forms $q(x_0,x_1,x_2)$ that
are not sums of squares (sos). Such a form has at most ten real
zeros.
We give a complete and explicit characterization of all sets
$S\subset\P^2(\R)$ with $|S|=9$ for which there is a psd non-sos
sextic vanishing in $S$.
Roughly, on every plane cubic $X$ with only real nodes there is a
certain natural divisor class $\tau_X$ of degree~$9$, and $S$ is the
real zero set of some psd non-sos sextic if, and only if, there is a
unique cubic $X$ through $S$ and $S$ represents the class $\tau_X$ on
$X$. If this is the case, there is a unique extreme ray
$\R_\plus q_S$ of psd non-sos sextics through $S$, and we show how to
find $q_S$ explicitly.
The sextic $q_S$ has a tenth real zero which for generic $S$ does not
lie in $S$, but which may degenerate into a higher singularity
contained in $S$. We also show that for any eight points in
$\P^2(\R)$ in general position there exists a psd sextic that is not
a sum of squares and vanishes in the given points.
\end{abstract}

%-------------------------------------------------------------------%

\section*{Introduction}

In a famous and influential paper, Hilbert \cite{hi88} proved in 1888
that a polynomial with real coefficients that takes nonnegative
values only can usually not be written as a sum of squares of real
polynomials. More precisely, there exist nonnegative ternary forms of
any even degree $\ge6$ that are not sums of squares. For forms in
four or more variables, the same is true for even degrees~$\ge4$.
Hilbert could well have used his own arguments to
construct concrete examples of such forms. However he didn't do so,
and it took almost 80 more years before first explicit examples
appeared in print.

It has become common use to say that a real form is \emph{psd}
(positive semidefinite) if it has nonnegative values, and that it is
\emph{sos} (sum of squares) if it is a sum of squares of forms. First
examples of psd forms that are not sos were found in the 1960s by
Motzkin, Ellison and Robinson, being verified in ad~hoc ways. In the
1970s, Choi, Lam and later Reznick started to explore the phenomenon
in more systematic ways.
In particular, they studied \emph{extreme} psd forms, which are psd
forms that are not squares and cannot be written as a sum of psd
forms in a nontrivial way. Closely related is the study of the
possible real zero sets of psd non-sos forms. On the other hand they
constructed a variety of new examples.
See Reznick's paper \cite{rz00} for a detailed account
with precise references of work done up to around the year 2000. In
\cite{rz07}, Reznick formalized the arguments from Hilbert's proof,
thereby showing in the case of ternary sextics that whenever two
plane cubics intersect in nine different $\R$-points, there exists a
psd non-sos sextic through any eight of them.

In this paper we focus entirely on ternary sextics. A psd sextic that
is not sos has at most ten real zeros. Our main contribution is a
complete and explicit characterization of all sets of nine points in
$\P^2(\R)$ on which some psd non-sos sextic vanishes.
To describe the result let $X$ be a reduced plane cubic, and let $S$
be a set of 9 nonsingular $\R$-points on $X$. We say that $S$ is
admissible if the class of the Weil divisor $\sum_{P\in S}P$ on $X$
is a non-trivial half of $[\scrO_X(6)]$, and if this class satisfies
a certain definiteness condition (\ref{dfntaux}). For every
admissible $S$ there exists a psd non-sos sextic that vanishes in
$S$. In fact, up to positive scaling there is a unique such sextic,
denoted $q_S$, that is extreme. Conversely, for every psd non-sos
sextic with a set $S$ of nine real zeros, there is a unique cubic $X$
through $S$, and $S$ is admissible in the sense before
(\ref{admiffexpsdnonsossext}).

As for the cubics that arise in this construction, we show that a
cubic $X$ contains an admissible set if and only if $X$ has at most
real nodes (with real tangents) as singularities. Moreover, the
divisor class in question is unique in this case
(\ref{def2torscubic}).

Our characterization is explicit enough to make it easy, for any
given set $S$, to check effectively whether $S$ is admissible, and in
the positive case, to write down concretely a psd non-sos sextic
through $S$, or even the extreme sextic $q_S$. In this way we obtain
an explicit parametrization of the extreme psd sextics with at least
nine real zeros. Note that these sextics are dense inside the set of
all extreme psd sextics. Since generically an extreme psd sextic has
ten real zeros, our parametrization happens in a 10:1 fashion.
Recently, Blekherman et~al.\ showed \cite{bhors} that the extreme psd
sextics form a Zariski dense subset of the Severi variety of all
rational sextics. In particular, they form a family of (projective)
dimension~$17$. Our approach gives a new and very clear approach to
these facts. We apply our result to prove that any 8 points of the
plane are the real zero set of a psd non-sos sextic, unless 4 of them
are on a line or 7 are on a conic (\ref{8ptsexpsdnonsos}). Before,
this was known only under additional assumptions.

The paper is organized as follows. In Section~2 we discuss Picard
groups of real (possibly singular) curves. We define definite
$2$-torsion classes in the Picard group and determine these classes
in the case of plane cubics. In Section~3 we show that every
set of nine real zeros of a psd non-sos sextic is admissible. The
converse is proved in Section~4, where we also state a number of
complements to the main result. Section~5 contains a series of examples
illustrating various aspects of our construction.

%-------------------------------------------------------------------%

\section{Preliminaries and notation}

We work with ternary forms only. Let $d\ge0$. By $P_{2d}$ (resp.\
$\Sigma_{2d}$) we denote the set of all psd (positive semidefinite)
forms $f\in\R[x_0,x_1,x_2]$ with $\deg(f)=2d$ (resp.\ its subset of
all sos forms, i.e.\ forms that can be written as a finite sum of
squares of forms). It is well known that $\Sigma_{2d}\subset P_{2d}$
are closed convex cones, and that the inclusion is proper if and only
if $2d\ge6$.

If $C$ is any convex cone with $C\cap(-C)=\{0\}$, the extreme rays of
$C$ are the one-dimensional faces of $C$. In the case of the psd cone
$P_{2d}$ we reserve the term for forms that are not sums of squares.
Thus, we will say that a psd form $q\in P_{2d}$ is \emph{extreme} if
$q\notin\Sigma_{2d}$, and if $q=q_1+q_2$ with $q_1,\,q_2\in P_{2d}$
implies $q_1,\,q_2\in\R_\plus q$.

All varieties $X$ (usually curves) are defined over $\R$ and are
considered as $\R$-schemes. As usual, $X_\sing$ resp.\ $X_\reg$
denotes the singular resp.\ the nonsingular locus of $X$, and
$X(\R)$ (resp.\ $X(\C)$) is the set of $\R$-rational (resp.\
$\C$-rational) points of $X$. For a form $f\in\R[x_0,x_1,x_2]$,
$V(f)$ is the curve $f=0$ in $\P^2$, and $V_\R(f)$ denotes the set
of its real points. When $f,\,g\in\R[x_0,x_1,x_2]$ are forms without
common irreducible component and $P\in\P^2(\R)$, the local
intersection number of $f$ and $g$ at $P$ is written $i_P(f,g)$.

Let $X$ be a curve over $\R$. A scheme point $x\in X$ is called a
\emph{node} (resp.\ an \emph{acnode}) of $X$ if $x$ has residue
field $\R$ and has multiplicity 2 with two different tangent
directions, where the tangents are real resp.\ complex conjugate.
Thus $x$ is a node iff $\wh\scrO_{X,x}\cong\R[[u,v]]/(uv)$, and is an
acnode iff $\wh\scrO_{X,x}\cong\R[[u,v]]/(u^2+v^2)$. The real curve
$X(\R)$ has two branches intersecting transversally at $x$ when $x$
is a node, and has an isolated point at $x$ when $x$ is an acnode.

For $S\subset\P^2(\R)$ a finite set, we denote by $I_d(mS)$ the space
of degree~$d$ forms in $\R[x_0,x_1,x_2]$ that have multiplicity
$\ge m$ at every point $P\in S$.

%-------------------------------------------------------------------%

\section{Picard groups of real curves}

We need to work with groups of divisor classes not only on
nonsingular (plane) curves, but also on singular and even reducible
curves. Therefore we need to discuss Picard groups in this
generality.

\begin{lab}
Let $X$ be a reduced projective curve over $\R$, always considered as
a scheme over $\R$. We allow $X$ to be reducible. By $X_\C$ we denote
the base field extension $X_\C=X\times_{\Spec(\R)}\Spec(\C)$. The
ring of rational functions on $X$ is $\R(X)=H^0(X,\scrK_X)$, and
similarly $\C(X)=H^0(X_\C,\scrK_{X_\C})$. The group of Cartier
divisors on $X$ is $\Div(X)=H^0(X,\scrK_X^*/\scrO_X^*)$, linear
equivalence of divisors on $X$ is denoted by~$\sim$. The Picard group
$\Pic(X)$ is (isomorphic to) the group of divisors on $X$ modulo
linear equivalence, i.e., the natural sequence
$$1\to\scrO(X)^*\to\R(X)^*\to\Div(X)\to\Pic(X)\to0$$
is exact.

For $U\subset X$ open there is a natural map $\Div(U)\to\Div(X)$, the
extension by zero. If $U$ is dense in $X$ then this map is surjective
up to linear equivalence. In particular, this applies to $U=X_\reg$,
the nonsingular locus of $X$. Since Cartier divisors on $X_\reg$ can
be identified with Weil divisors (zero cycles) on $X_\reg$, we will
often tacitly represent divisor classes on $X$ by Weil divisors
on~$X_\reg$.
\end{lab}

\begin{lab}
Assume from now on that $X$ is geometrically connected and that the
set $X(\R)$ of $\R$-rational points is Zariski dense in $X$. The
Galois group $G=\Gal(\C/\R)$ acts on $\Pic(X_\C)$, and the
Hochschild-Serre sequence $H^i(G,\,H^j_\et(X_\C,\G_m))$ $\To$
$H^{i+j}_\et(X,\G_m)$ implies that the natural map $\Pic(X)\to
\Pic(X_\C)^G$ is an isomorphism.

Let $X_1,\dots,X_m$ be the irreducible components of $X$, and let $J$
be the generalized Jacobian of $X$. Then $J$ is a connected algebraic
group over $\R$, and we have the exact sequence
$$0\to J(\C)\to\Pic(X_\C)\to\Z^m\to0$$
where the last map is given by the partial degrees
\cite{blr}.
The Galois action on this sequence gives the exact sequence
$$0\to J(\R)\to\Pic(X)\to\Z^m\to0$$
(the last map is surjective since every component $X_i$ contains a
nonsingular $\R$-point of $X$) and the isomorphism
$$H^1(\R,J)\isoto H^1(G,\,\Pic(X_\C))$$
(writing $H^1(\R,J):=H^1(G,J(\C))$ as usual).
\end{lab}

Given an abelian group $M$, we denote by ${}_nM=\ker(M\labelto nM)$
the $n$-torsion subgroup of $M$, for $n\in\N$. For any $G$-module $M$
there is a natural map ${}_2M^G\to H^1(G,M)$. We consider this map
for $M=\Pic(X_\C)$:

\begin{prop}\Label{2pich1pic}%
The natural map ${}_2\Pic(X)\to H^1(G,\,\Pic(X_\C))$ is surjective.
\end{prop}

\begin{proof}
The map ${}_2J(\R)\to H^1(\R,J)$ is surjective, see \cite{sch:tams},
proof of Lemma 2.3(b).
So the assertion follows from the commutative diagram
$$\begin{xy}\xymatrix{%
{}_2J(\R) \ar[r] \ar[d] & H^1(\R,J) \ar[d] \\
{}_2\Pic(X) \ar[r] & H^1(G,\,\Pic(X_\C))
}\end{xy}$$
in which the right hand vertical arrow is surjective, see above.
\end{proof}

\begin{lab}\Label{dfnpsi}%
Let $C(X)$ denote the set of connected components of $X(\R)$, and let
the finite abelian group $A_X$
be defined by the exact sequence
$$1\to\{\pm1\}\to\{\pm1\}^{C(X)}\to A_X\to1$$
where the first map is the diagonal embedding. We define a
homomorphism
$$\psi\colon H^1(G,\,\Pic(X_\C))\>\to\>A_X$$
as follows (cf.\ \cite{sch:tams} 2.2). Given a Weil divisor $D$ on
$X_\C$ with nonsingular support and with $D+\ol D\sim0$, there is a
rational function $g\in\C(X)^*$ with $\div(g)=D+\ol D$, and $g$ can
be chosen to lie in $\R(X)^*$.
Therefore $g$ has even order in every point of $X_\reg(\R)$ and is
invertible around $X_\sing(\R)$. So there exists a sign tuple
$\epsilon\in\{\pm1\}^{C(X)}$ such that $\epsilon(\xi)g(\xi)\ge0$ for
every $\xi\in X(\R)$ where $g$ is defined. Since $g$ depends on $D$
only up to a factor in $\scrO(X)^*=\R^*$, this gives a well-defined
element of $A_X$, which does not change if $D$ gets replaced by an
equivalent Weil divisor on $(X_\C)_\reg$.
So the map $\psi$ that sends the class of $D$ to $\pm\epsilon\in A_X$
is a well-defined homomorphism.
\end{lab}

\begin{prop}\Label{psisurj}%
The map $\psi\colon H^1(G,\,\Pic(X_\C))\to A_X$ is surjective.
\end{prop}

\begin{proof}
Let a sign distribution $\epsilon\in\{\pm1\}^{C(X)}$ be given. By
Weierstra\ss\ approximation there exists a rational function
$f\in\R(X)^*$ without real zeros or poles that is invertible around
every singular point of $X$, and such that $\epsilon(\xi)f(\xi)\ge0$
holds for every $\xi\in X(\R)$.
The divisor $\div(f)$ of $f$ on $X_\C$ is a $G$-invariant Weil
divisor on $(X_\C)_\reg$ without real points. Hence it can be written
$\div(f)=D+\ol D$ with $D\in\Div(X_\C)$. The class defined by $D$ in
$H^1(G,\,\Pic(X_\C))$ maps to $\pm\epsilon$ under~$\psi$.
\end{proof}

When $X$ is nonsingular, the map $\psi$ in \ref{psisurj} is even
bijective, see \cite{sch:tams} Remark~2.2.

\begin{cor}\Label{sigmasurj}%
The map $\sigma\colon{}_2\Pic(X)\to A_X$ which is the composite of
the map \ref{2pich1pic} and of $\psi$ (see \ref{dfnpsi}) is
surjective.
\qed
\end{cor}

\begin{dfn}\Label{dfndef}%
A $2$-torsion class $\tau\in{}_2\Pic(X)$ will be called
\emph{definite} if it lies in the kernel of the natural map
$\sigma\colon{}_2\Pic(X)\to A_X$ (see \ref{sigmasurj}).
\end{dfn}

So if $D$ is a Weil divisor on $X$ with nonsingular support and with
$2D\sim0$, the class $[D]$ is definite if and only if there exists a
rational function $g\in\R(X)^*$, invertible around $X_\sing$, with
$\div(g)=2D$ and with $g\ge0$ on $X(\R)$ (where $g$ is defined).

For plane cubics we determine the definite $2$-torsion classes
explicitly:

\begin{prop}\Label{def2torscubic}%
Let $X$ be a reduced plane cubic curve with $X(\R)$ Zariski dense in
$X$. If every singularity of $X$ is a node, there exists a unique
nonzero class in ${}_2\Pic(X)$ that is definite. Otherwise the only
definite class in ${}_2\Pic(X)$ is zero.
\end{prop}

\begin{proof}
When $X$ is nonsingular, the assertion says $|\ker(\sigma)|=2$ and is
a particular case of \cite{sch:tams}, Lemma 2.3(c). (The map $\sigma$
considered here is identified with the map $\bar\phi$ studied in
\emph{loc.\,cit.}, via the natural isomorphism $H^1(\R,J)\isoto A_X$
observed in \cite{sch:tams} Remark~2.2.)

For singular $X$ the claim is proved by computing the group $J(\R)=
\Pic_0(X)$ and using \ref{sigmasurj}. In more detail (compare
\cite{blr} 9.2 for computation of the Picard groups):

When $X$ is irreducible, $\Pic_0(X)=J(\R)$ is isomorphic to $\R$,
$\R^*$ or $\C^*/\R^*$, depending on whether $X$ has a cusp, a node
or an acnode, respectively. Note that $|A_X|=1$ in the first two
cases and $|A_X|=2$ in the third.
When $X$ is the union of a conic $Q$ and a line $L$, we have
$\Pic_0(X)\cong\R^*$, $\R$ or $\C^*/\R^*$, depending on whether
$|Q(\R)\cap L(\R)|=2$, $1$ or $0$. Again we have $|A_X|=1$ in the
first two cases and $|A_X|=2$ in the last.
When $X$ is the union of three lines, $\Pic_0(X)\cong\R$ or $\R^*$,
according to whether the lines meet in a common point or not, and
$|A_X|=1$ in either case. Since $\sigma$ is surjective by Corollary
\ref{sigmasurj}, this implies the assertion in each case.
\end{proof}

\begin{notat}\Label{dfntaux}%
Let $X$ be a reduced plane cubic with $X(\R)$ Zariski dense in $X$
whose only singularities are nodes, and let $\tau$ be the unique
definite nonzero $2$-torsion class in $\Pic(X)$ (see Proposition
\ref{def2torscubic}). We write
$$\tau_X\>:=\>\tau+[\scrO_X(3)]\>\in\Pic(X).$$
\end{notat}

%-------------------------------------------------------------------%

\section{Nine real zeros}

Let $S\subset\P^2(\R)$ be any set of nine points. There is at least
one cubic $X$ passing through $S$, i.e., $\dim I_3(S)\ge1$. If there
exists a psd and non-sos sextic through $S$, we show in this section
that $X$ is unique, and that $S$ gives rise to the divisor class
$\tau_X$ on $X$ (see \ref{dfntaux}). In the next section we will
prove the converse.

\begin{lab}
For $T\subset\P^2(\R)$ a finite set, recall that $I_d(mT)\subset
\R[x_0,x_1,x_2]$ is the space of forms of degree $d$ that have
multiplicity $\ge m$ at every point of $T$. The psd forms of
degree~$2d$ that vanish in $T$ form a face of the cone $P_{2d}$
denoted
$$P_{2d}(T)\>=\>\{p\in P_{2d}\colon p|_T=0\},$$
and we put $\Sigma_{2d}(T)=\Sigma_{2d}\cap P_{2d}(T)$. The forms in
$\Sigma_{2d}(T)$ are the sums of squares of forms in $I_d(T)$.
\end{lab}

The following facts are certainly well known, and are recorded for
reference:

\begin{prop}\Label{rappel}%
Let $T\subset\P^2(\R)$ be any set of 8 points, no 4 on a line and no
7 on a conic.
\begin{itemize}
\item[(a)]
$\dim I_3(T)=2$, and almost every cubic in $I_3(T)$ is nonsingular.
\item[(b)]
$\dim I_6(2T)=4$ and $\dim I_3(T)^2=3$.
\item[(c)]
Every psd sextic in $I_3(T)^2$ is a sum of two squares of cubics. In
particular, $\Sigma_6(T)=I_3(T)^2\cap P_6(T)$.
\end{itemize}
\end{prop}

Here and in the sequel, $I_3(T)^2$ denotes the space of sextic forms
spanned by the products $f_1f_2$ with $f_1,\,f_2\in I_3(T)$.

\begin{proof}
(a) is classical, see e.g.\ \cite{ha} V.4.4. (b) Let $I_3(T)=
\spn(f,f')$ with $f$ nonsingular. Then $I_3(T)^2=\spn(f^2,ff',f'^2)$
has dimension~$3$ since $f,\,f'$ are algebraically independent.
Being singular in the points of $T$ imposes $8\cdot3=24$ linear
conditions on a sextic, showing $\dim I_6(2T)\ge4$. For the reverse
inequality consider the divisor $D=\sum_{P\in T}P$ on $X=V(f)$, and
let $M\in X(\R)$ be the point with $\div_X(f')=D+M$. If $g\in
I_6(2T)$ is prime to $f$ and satisfies $\div_X(g)\ge2D+M$ then
$\div_X(g)=2(D+M)$, implying that $g\in I_3(T)^2$. This shows that
$I_3(T)^2$ has codimension $\le1$ in $I_6(2T)$.
(c) Let $g\in I_3(T)^2$ be a psd sextic form. Then $g=q(f,f')$ is a
quadratic form in $f$ and $f'$, and the binary quadratic form
$q(t_0,t_1)$ is necessarily positive definite, as can be seen locally
around any transversal intersection point of $f$ and $f'$.
Hence $q$ is a sum of two squares from $I_3(S)$.
\end{proof}

\begin{lab}\Label{assmpts}%
Let $g\in\R[x_0,x_1,x_2]=\R[\x]$ be a form of degree~$6$ that is psd.
If $g$ is reducible over $\R$ then $g$ is sos, since every ternary
psd form of degree $\le4$ is sos (Hilbert \cite{hi88}).
Moreover $|V_\R(g)|\le5$ or $|V_\R(g)|=\infty$ in this case.
Now assume that $g$ is irreducible over~$\R$.
Since every real zero of $g$ is a singular point of the curve $V(g)$,
there can be at most 10 real zeros (and at most~$9$ when $g$ is
reducible over~$\C$).
We assume that $g$ has at least 9 real zeros, and we fix a set
$S\subset V_\R(g)$ with $|S|=9$.
\end{lab}

\begin{lem}\Label{anycubic}%
Let $g\in\R[\x]$ be a psd sextic form with $|V_\R(g)|<\infty$, and
let $S\subset V_\R(g)$ with $|S|=9$. Any cubic through $S$ is reduced
with Zariski dense $\R$-points, and is nonsingular in the points
of~$S$. Moreover,
\begin{itemize}
\item[(a)]
if $\dim I_3(S)=2$ then $g$ is a sum of two squares of cubic forms,
hence reducible over~$\C$,
\item[(b)]
otherwise $\dim I_3(S)=1$, and $g$ is absolutely irreducible and not
a sum of squares in $\R[\x]$.
\end{itemize}
\end{lem}

\begin{proof}
$g$ is irreducible, so $S$ contains no $4$ points on a line and no
$7$ points on a conic.
By \ref{rappel}(a) we have $\dim I_3(S)\in\{1,2\}$.

Let $X=V(f)$ be a cubic through $S$. It is clear that $X$ is reduced
and $X(\R)$ is Zariski dense in $X$.
Moreover $S$ cannot contain a singular point of $X$, since otherwise
the intersection product of $X$ and $Y=V(g)$ would satisfy $X\mal Y>
2\cdot9=18$, contradicting that $Y$ is irreduzible.

If $\dim I_3(S)=2$ and $I_3(S)$ is spanned by $f$ and $f'$, then
$f,\,f'$ are relatively prime and $S$ is the set of their common
zeros.
It follows that $g$ is a quadratic form in $f,\,f'$, and by
\ref{rappel}(c), $g$ is a sum of two squares from $I_3(S)$.

On the other hand, when $X=V(f)$ is the unique cubic through $S$,
the only sos sextics vanishing in $S$ are multiples of $f^2$. So $g$
is not a sum of squares in this case, and is therefore irreducible
over~$\C$.
\end{proof}

\begin{lab}\Label{intfall}%
Assume now that the psd sextic $g$ is not a sum of squares and has at
least 9 real zeros.
Fix a subset $S\subset V_\R(g)$ with $|S|=9$. By \ref{anycubic} there
is a unique cubic $X=V(f)$ through $S$, and $X$ is reduced with
$X(\R)$ Zariski dense in $X$ and $S\subset X_\reg(\R)$. However, $X$
may be singular or even reducible. The Weil divisor
$D=\sum_{P\in S}P$ on $X_\reg$ defines a class in $\Pic(X)$. Let
$L\in\Div(X)$ be the divisor of a line section of $X$.
\end{lab}

\begin{prop}\Label{2torsne0}%
The divisor class $[D-3L]$ in $\Pic(X)$ is a nonzero $2$-torsion
class and is definite (\ref{dfnpsi}).
\end{prop}

So $[D]=\tau_X$, see Notation \ref{dfntaux}.

\begin{proof}
From $\div_X(g)\sim6L$ and $\div_X(g)=2D$
it follows that $2(D-3L)\sim0$ on $X$. Assuming $D\sim3L$ there would
exist a cubic form $f'$, prime to $f$, with $\div_X(f')=D$,
contradicting $I_3(S)=\R f$ (\ref{anycubic}). Since $2(D-3L)$ is the
divisor of the psd rational function $g/l^6$ on $X$, it is clear that
$\tau$ is definite, see Definition \ref{dfndef}.
\end{proof}

For a convenient way of speaking we introduce the following
terminology:

\begin{dfn}\Label{dfnadm}%
A set $S\subset\P^2(\R)$ will be said to be \emph{admissible} if
$|S|=9$, and if there exists a reduced plane cubic $X$ with $S\subset
X_\reg(\R)$ for which $\sum_{P\in S}[P]=\tau_X$ in $\Pic(X)$.
\end{dfn}

We have therefore proved:

\begin{cor}\Label{zsfgpsdsex}%
Let $g$ be a psd sextic that is not a sum of squares. Then any set
$S\subset V_\R(g)$ with $|S|=9$ is admissible.
\qed
\end{cor}

Next we discuss properties of admissible sets.

\begin{prop}\Label{propsadm}%
Let $S\subset\P^2(\R)$ be an admissible set, let $f\ne0$ be a cubic
form vanishing on $S$, and let $X=V(f)$.
\begin{itemize}
\item[(a)]
$I_3(S)=\R f$, so $X$ is uniquely determined by $S$.
\item[(b)]
$X$ is reduced, $S\subset X_\reg(\R)$, and $X(\R)$ is Zariski dense
in $X$.
\item[(c)]
For every irreducible component $X'$ of $X$ we have $|S\cap X'(\R)|
=3d'$ where $d'=\deg(X')$.
\end{itemize}
In view of (a), we call $X=V(f)$ the cubic \emph{associated} to~$S$.
\end{prop}

\begin{proof}
Since $S$ is admissible, there exists a reduced cubic as in
Definition \ref{dfnadm}. In view of (a), it suffices to prove
(a)--(c) for this cubic.
So we may assume that $X=V(f)$ is reduced with $S\subset X_\reg(\R)$,
and that $[\scrO_X(3)]-\sum_{P\in S}[P]$ is a $2$-torsion class in
$\Pic(X)$. Let $X'\subset X$ be an irreducible component, and let
$D'=\sum_{P\in S\cap X'(\R)}P\in\Div(X')$. Then $2[D']=
[\scrO_{X'}(6)]$ in $\Pic(X')$, which implies (c). In particular,
every irreducible component of $X$ contains a nonsingular $\R$-point.
Hence $X(\R)$ is Zariski dense in $X$.
It remains to prove (a). Let $h\in I_3(S)$, and assume $h\notin\R f$.
Writing $g=\gcd(f,h)$ and $f=f'g$, $h=h'g$, the degree $d':=\deg(f')=
\deg(h')$ satisfies $d'\ge1$. Put $X'=V(f')$ and $S'=S\cap X'(\R)$.
Each point of $S$ lies on a single irreducible component of $X$,
therefore $h'$ vanishes on $S'$.
Therefore by (c),
$h'$ has at least $3d'$ different zeros on $X'$. On the other hand
the number of zeros of $h'$ on $X'$ is $d'^2$ by B\'ezout's theorem,
since $\gcd(f',h')=1$. It follows that $d'=3$, and so $\gcd(f,h)=1$.
Hence $D=\div_X(h)$,
which implies $[D]=[\scrO_X(3)]$, contradicting $[D]=\tau_X$. So the
assumption was false, proving~(a).
\end{proof}

\begin{rem}\Label{4poss}%
Let $X$ be the cubic associated to the admissible set $S$. According
to Proposition \ref{def2torscubic}, there are four possibilities for
$X$: Either $X$ is nonsingular, or else irreducible with a node, or
else the union of a line and a conic intersecting transversally in
two $\R$-points, or else a triangle (union of three lines without a
common point).
\end{rem}

\begin{rems}\Label{admissgeom}%
The admissibility condition can be characterized in more elementary
geometric terms, as follows. Let $X=V(f)$ be a plane cubic curve, and
assume for simplicity that $X$ is irreducible.
\smallskip

1.\
Let $S\subset X_\reg(\R)$ be a set with $|S|=9$.
Choose a point $P\in S$ (it does not matter which one), and let
$X'=V(f')$ be a cubic through $S'=S\setminus\{P\}$, different from
$X$.
Let $M$ be the 9th point of intersection of $X$ and $X'$, a
nonsingular point of $X$, and let $t_M$ resp.\ $t_P$ be the tangent
of $X$ at $M$ resp.~$P$.
Then $S$ is admissible if and only if $M\ne P$, the tangents $t_M$
and $t_P$ meet in a point of $X$, and the product $t_Mt_P$ is
semidefinite on $X(\R)$.

Indeed, by \ref{propsadm}(a), $S$ is admissible if and only if $X$
and $S$ satisfy the property in \ref{dfnadm}. Consider the divisors
$D=\sum_{Q\in S}Q$ and $D'=M+\sum_{Q\in S'}Q$ on $X$, and let
$L=\div(l)\in\Div(X)$ be the divisor of a linear form~$l$ with
$l\nmid f$. Then $D'\sim3L$, so $D-3L\not\sim0$ means $M\ne P$.
On the other hand,
$2(D-3L)\sim0$ means $2M\sim2P$ on $X$, which says that $t_M$ and
$t_P$ meet on $X$.
Assuming this condition, $\frac{t_P}{t_M}$ is a rational function on
$X$ whose divisor is $2(P-M)$. So the rational function
$$\frac{f'^2}{l^6}\cdot\frac{t_P}{t_M}$$
on $X$ has divisor $2(D-3L)$,
proving the characterization of the definiteness condition claimed
above.
\smallskip

2.\
Let $X$ be an irreducible plane cubic and $T\subset X_\reg(\R)$ a set
with $|T|=8$. There can be at most one point $Q\in X(\R)$ for which
the set $T\cup\{Q\}$ is admissible. In fact, such $Q$ exists if and
only if $X$ is nonsingular or has a node, and if the unique point
$Q\in X_\reg(\R)$ with $[Q+\sum_{P\in T}P]=\tau_X$ satisfies $Q\notin
T$. In this case, $T\cup\{Q\}$ is admissible.

The previous remark shows how to construct $Q$ geometrically. Indeed,
let $M$ be as before, and let $N$ be the third intersection point of
$t_M$ with $X$. When $X$ is nonsingular or has a node, there is a
unique tangent $t\ne t_M$ to $X$ that passes through $N$ for which
$t_Mt$ is semidefinite on $X(\R)$. This tangent $t$ touches $X$
in~$Q$.
\end{rems}

%-------------------------------------------------------------------%

\section{Constructing nonnegative sextics with nine real zeros}

In this section we prove the converse of Corollary \ref{zsfgpsdsex}.
As before we work in the projective plane $\P^2$ over $\R$ and write
$\x=(x_0,x_1,x_2)$.

We start with two technical lemmas.

\begin{lem}\Label{singularize0}%
Let $X=V(f)$ be a reduced plane cubic, and let $T\subset X_\reg(\R)$
be a set with $|T|=8$ or~$9$. If $|T|=9$, assume that $I_3(T)=\R f$.
Let $g$ be a sextic form with $\gcd(f,g)=1$ such that $i_P(f,g)\ge2$
for every $P\in T$. Then there exists a cubic form $p$ such that the
sextic $g+pf$ is singular in every point of $T$.
\end{lem}

\begin{proof}
The sextic $g$ shows that no $4$ points of $T$ are on a line and no
$7$ on a conic.
Therefore $\dim I_3(T)=2$ in the case $|T|=8$ (\ref{rappel}). Let
$f_{x_i}:=\partial f/\partial x_i$ ($i=0,1,2$) denote the partial
derivatives. From the assumptions it follows that
$$\rk\begin{pmatrix}f_{x_0}&f_{x_1}&f_{x_2}\\g_{x_0}&g_{x_1}&g_{x_2}
\end{pmatrix}(P)\ =\ 1$$
at every point $P\in T$.
Hence there exists a unique section $\lambda\in\Gamma(T,\scrO(3))$
such that $g_{x_i}-\lambda f_{x_i}$ vanishes on $T$ ($i=0,1,2$).
The natural map $\Gamma(\P^2,\scrO(3))\to\Gamma(T,\scrO(3))$ is
surjective since its kernel $I_3(T)$ has dimension $10-|T|$. Hence
there exists a cubic form $p$ which restricts to $\lambda$ on $S$.
The sextic $q-pf$ is singular in the points of $T$.
\end{proof}

\begin{lem}\Label{makelocalpsd}%
Let $X=V(f)$ be a plane curve of degree $d$, let $g$ be a form of
degree $2d$, and let $P\in X_\reg(\R)$. Suppose that $g$ is singular
at $P$ and $i_P(f,g)=2$. If $g\ge0$ locally around $P$ on $X(\R)$,
there exists $t>0$ such that $g+tf^2\ge0$ locally around $P$ in the
plane.
\end{lem}

\begin{proof}
Given the other assumptions, the condition $i_P(f,g)=2$ means that
the tangent of $X=V(f)$ at $P$ is not a tangent of the curve $V(g)$
at~$P$. Therefore, for sufficiently large $t>0$, $g+tf^2$ has
positive definite Hessian at $P$ in local affine coordinates, proving
the lemma.
\end{proof}

In the following let always $S\subset\P^2(\R)$ be an admissible set
(Definition \ref{dfnadm}), and let $X=V(f)$ be the cubic associated
to $S$, see \ref{propsadm}. Writing $D=\sum_{P\in S}P\in\Div(X)$ we
have $[D]=\tau_X$.

\begin{lem}\Label{singularize}%
There exists a sextic form $q$ with $\gcd(f,q)=1$ that is singular in
the points of $S$.
\end{lem}

\begin{proof}
Since $2[D]=[\scrO_X(6)]$, there exists a sextic form $g$ with
$\gcd(f,q)=1$ such that $\div_X(q)=2D$.
We can apply Lemma \ref{singularize0} since $I_3(S)=\R f$ holds by
\ref{propsadm}(a).
This proves the lemma.
\end{proof}

\begin{lem}\Label{expsdsextic}%
Let $q$ be a sextic as in Lemma \ref{singularize}. There exists a
real number $t>0$ and a choice of sign $\pm$ such that $tf^2\pm q$ is
psd (on $\P^2(\R)$).
\end{lem}

\begin{proof}
The sextic $q$ has $\div_X(q)=2D$. Let $l$ be a linear form not
dividing $f$, and let $L=\div(l)\in\Div(X)$. Since the class
$[D-3L]=\tau_X\in{}_2\Pic(X)$ is definite by assumption, the rational
function $q/l^6$ on $X$ is semidefinite on $X(\R)$.
After replacing $q$ with $-q$ if necessary, we can assume that
$q\ge0$ on $X(\R)$. For any point $P\in S$ the local intersection
number $i_P(f,q)=2$, so by Lemma \ref{makelocalpsd} there exists
$t>0$ such that $tf^2+q$ is psd around~$P$. Now the assertion follows
from the following Lemma \ref{lem1}, which is an easy compactness
argument.
\end{proof}

\begin{lem}\Label{lem1}%
Let $p,\,q\in\R[\x]$ be forms of the same even degree. Assume that
$p$ is psd, and that for every real zero $P$ of $p$ there is a real
number $t$ such that $q+tp$ is nonnegative in a neighborhood of $P$.
Then there exists $t\in\R$ such that the form $q+tp$ is psd.
\end{lem}

\begin{proof}
Since $\P^2(\R)$ is compact, there exist finitely many open sets
$U_i\subset\P^2(\R)$ and real numbers $t_i$ ($i=1,\dots,r$), such
that $\P^2(\R)$ is covered by the $U_i$ and $q+t_ip\ge0$ on $U_i$ for
every~$i$. It suffices to take $t=\max\{t_1,\dots,t_r\}$.
\end{proof}

The next result completes and summarizes the discussion:

\begin{thm}\Label{sumcomp}%
Let $S\subset\P^2(\R)$ be an admissible set with associated cubic
$X=V(f)$ (see \ref{propsadm}). The space $I_6(2S)$ of sextics
singular in $S$ has dimension~$2$. The psd sextics vanishing in $S$
form a $2$-dimensional cone $P_6(S)$ in $I_6(2S)$, while $\Sigma_6(S)
=\R_\plus f^2$ has dimension~$1$. There exists a unique (up to
positive scaling) sextic in $P_6(S)\setminus\Sigma_6(S)$ that is
extreme in the psd cone.
\end{thm}

\begin{proof}
By Lemma \ref{expsdsextic} there exists a psd sextic $q$ with
$(f,q)=1$ that vanishes in $S$. Such $q$ cannot be a sum of squares
since $q=\sum_iq_i^2$ implies $q_i\in I_3(S)$, and since $I_3(S)=
\R f$ by \ref{propsadm}(a).
So it only remains to show that
$I_6(2S)$ is spanned by $q$ and $f^2$.
The argument is slightly technical since $X$ may be reducible.
Let $0\ne h\in I_6(2S)$, let $f_1=\gcd(f,h)$, and write $f=f_1f_2$
and $h=f_1h_2$. Let $d_i=\deg(f_i)$ ($i=1,2$), then we have
$d_1+d_2=3$ and $\deg(h_2)=6-d_1=3+d_2$. Let $X_2:=V(f_2)$ and
$S_2:=S\cap X_2(\R)$. Then $|S_2|=3d_2$ by \ref{propsadm}(c). Since
$\gcd(f_2,h_2)=1$, the form $h_2$ has precisely $(6-d_1)d_2$ zeros on
$X_2$ by B\'ezout. On the other hand, $h_2$ has multiplicity $\ge2$
at each point of $S_2$,
and so $h_2$ has at least $2|S_2|=6d_2$ zeros on $X_2$. It follows
that $d_1=0$ or $d_2=0$. If $d_1=0$ then $\gcd(f,h)=1$, and
we conclude $\div_X(h)=2D=\div_X(q)$, so $h=cq+ff'$ (mod~$f$) with
some $c\in\R^*$ and some cubic form $f'$.
Necessarily $f'\in I_3(S)$, so
$h\in\R q+\R f^2$. On the other hand, if $d_2=0$ then $h=ff'$
with some cubic form $f'$, and using again the argument just given we
conclude $h\in\R f^2$.
\end{proof}

\begin{cor}\Label{admiffexpsdnonsossext}%
A set $S\subset\P^2(\R)$ with $|S|=9$ is admissible if and only if
there exists a psd sextic vanishing in $S$ that is not a sum of
squares.
\end{cor}

\begin{proof}
Follows from \ref{zsfgpsdsex} and \ref{sumcomp}.
\end{proof}

\begin{notat}\Label{nontatgs}%
Let $S\subset\P^2(\R)$ be an admissible set. We denote the unique (up
to scaling)
extreme psd form in $P_6(S)$ by $q_S$, bearing in mind the ambiguity
by a scalar factor.
\end{notat}

\begin{rem}
In the situation of \ref{sumcomp}, fix any $q\in I_6(2S)$ with
$q\notin\R f^2$. Replacing $q$ with $-q$ if necessary we can assume
$q\ge0$ on $X(\R)$. Then there exists a minimal real number $s$ for
which the form $q_s:=q+sf^2$ is psd, by Lemma \ref{expsdsextic} and
since the psd cone is closed. We have $q_s=q_S$, and the cone
$P_6(S)$ is generated by $f^2$ and $q_S$. The number $s$ resp.\ the
extreme form $q_S$ can be found numerically by solving a
semidefinite program, since $q_t$ is psd if and only if there exists
a quadratic form $p\ne0$ for which $pq_t$ is a sum of squares
(Hilbert \cite{hi93}).
Other ways of finding $q_S$ are discussed below (Remark
\ref{find10thzero}, Corollary \ref{9thzerononic} and Remark
\ref{remfind10thzero}).
\end{rem}

\begin{prop}\Label{10thzero}%
Let $S$ be admissible with associated cubic $X=V(f)$, let
$I_6(2S)=\spn(f^2,q)$ with $q\ge0$ on $X(\R)$. Write $q_t:=q+tf^2$
for $t\in\R$, and let $s\in\R$ be the minimal number for which $q_s$
is psd (so $q_s=q_S$).
\begin{itemize}
\item[(a)]
For each point $P\in S$ there exists a unique number $t(P)\in\R$ for
which $P$ is neither a node nor an acnode of the sextic $q_{t(P)}=0$.
\item[(b)]
$s\>\ge\>\max\{t(P)\colon P\in S\}$.
\item[(c)]
For generic choice of $S$ on $X$ the inequality (b) is strict, and
the extreme psd sextic $q_S$ has a tenth real zero $P$ with $P\notin
X(\R)$.
\item[(d)]
The sextic $q_S=0$ is a (geometrically) rational curve.
\end{itemize}
\end{prop}

\begin{proof}
(a)
Let $x,\,y$ be local affine coordinates centered at $P$ such
that $X=V(f)$ has tangent $x=0$ at $P$. Fix a linear form $l$ with
$l(P)\ne0$ such that $f/l^3=x+\text{(higher}$ order terms) at $P$.
Then $q/l^6=ax^2+bxy+cy^2+\text{(higher}$ order terms) at $P$ with
suitable $a,b,c\in\R$, and $c>0$ since $f$ and $q$ have no common
tangent at~$P$ and $q\ge0$ on $X(\R)$. Therefore $t(P)$ is determined
by the condition $b^2-4c(a+t(P))=0$. The singularity $P$ of $q_t$ is
a node for $t<t(P)$ and an acnode for $t>t(P)$.
Therefore $q_t$ is indefinite for $t<t(P)$, proving (a) and~(b).
Generically, $P$ will be a cusp of the sextic $q_{t(P)}$, and thus
$q_{t(P)}$ will be indefinite, showing $s>\max\{t(P)\colon P\in S\}$.
Therefore $q_s=q_S$ must have an additional tenth real zero in this
case, proving~(c). In the non-generic case $|V_\R(q_S)|=9$, one of
the points $P\in S$ is an $A_3$-singularity (with complex conjugate
branches) of $q_S$. In either case it is clear that $q_S$ is a
rational curve.
\end{proof}

Note that assertion (d) was already known from results in
\cite{bhors}, see Remark \ref{rationalsexticsrems}.2 below.

\begin{rem}\Label{find10thzero}%
Let $S$ be an admissible set. We describe an effective algorithm to
(precisely) compute the extreme sextic $q_S$, and possibly its 10th
zero. Consider the assumptions of Proposition \ref{10thzero}. For
almost all values of $t\in\C$, the sextic $q_t$ has a complex node at
each point of $S$, and no other singularities. Let $\scrE$ be the
finite set of numbers $t\in\C$ for which $q_t$ has a higher
singularity at some point of $S$, or a singularity outside $S$. The
set $\scrE$ can be found explicitly via elimination theory, e.g.\
using a computer algebra system. Clearly, $\scrE$ contains the
numbers $s$ and $t(P)$ from \ref{10thzero}. In fact, it is easy to
see that $s$ is precisely the largest real number in $\scrE$. Indeed,
this follows from \ref{10thzero}(b), since for real $t>s$ the sextic
$q_t$ is irreducible and cannot have an additional singularity
$P\notin S$ (necessarily $P$ would have to be real).
\end{rem}

\begin{rem}
``Hilbert's method'' for constructing psd non-sos sextics, as
formalized in \cite{rz07}, starts with a set $T\subset\P^2(\R)$ of
$8$ points such that the pencil $I_3(T)$ has a 9th base point
$E\notin T$. Choose a basis $f_1,\,f_2$ of $I_3(T)$ and a sextic
$g\in I_6(2T)$ with $g\notin I_3(T)^2$, i.e.\ $g(E)\ne0$ (c.f.\
\ref{rappel}). Then for large $t>0$, the sextic $q=g+t(f_1^2+f_2^2)$
is psd and not sos. Generically, $q$ has a 9th real zero $P$, but
there is no control of $P$, and $q$ is usually not extreme, not even
when $t$ is chosen minimally.
\end{rem}

We now consider collections of 8 points in the plane.

\begin{thm}\Label{8ptsexpsdnonsos}%
Let $T\subset\P^2(\R)$ be any set with $|T|=8$ and no 4 points on a
line, no 7 on a conic. Then there exists a psd non-sos sextic that
vanishes in $T$.
\end{thm}

This extends results by Reznick. He used Hilbert's construction to
prove the assertion when $T$ is ``copacetic'' (meaning that the 9th
base point of $I_3(T)$ lies outside $T$), and derived from this the
unconditional assertion for $|T|=7$. See \cite{rz07}, Corollaries 4.2
and 4.4. Our proof doesn't use these results.

\begin{proof}
Let $I_3(T)$ be spanned by the nonsingular cubics $f$ and $f'$, and
let $E$ be their 9th intersection point (which may lie in $T$). All
divisors will be formed on the curve $X=V(f)$. Let $Q\in X(\R)$ be
the unique point with $[Q+\sum_{P\in T}P]=\tau_X$. When $Q\notin T$
then $T\cup\{Q\}$ is admissible, see Remark \ref{admissgeom}.2.
Hence, in this case, there exists a psd non-sos sextic vanishing even
in $T\cup\{Q\}$.

So we assume $Q\in T$ and write $T'=T\setminus\{Q\}$. Note that
$Q\ne E$.
Let $D=\sum_{P\in T'}P\in\Div(X)$. We have $\div(f')=
D+Q+E$ and $[D+2Q]=\tau_X$. There exists a sextic form $g$, prime to
$f$, with $\div(g)=2(D+2Q)=2D+4Q$, and we have $g\ge0$ on $X(\R)$ by
the definiteness condition on $\tau_X$. By Lemma \ref{singularize0}
there exists a cubic form $h$ such that $g+fh$ is singular in every
point of~$T$. Replacing $g$ with $g+fh$ we can assume $g\in I_6(2T)$.
The divisor $\div(g)$ shows that $g\notin I_3(T)^2$.
In particular, adding to $g$ a form in $I_3(T)^2$ to $g$ cannot turn
$g$ into a sum of squares.

Since $f$ and $g$ have local intersection index~$2$ at every point of
$T'=T\setminus\{Q\}$, there exists $c>0$ such that $g':=g+cf^2$ is
locally psd around every point of $T'$, by Lemma \ref{makelocalpsd}.
For the form $g'':=g'+f'^2$, the intersection index with $f$ at the
point $Q$ drops to~$2$.
Now we can apply Lemma \ref{makelocalpsd} at the point $Q$, and then
apply Lemma \ref{lem1}, to conclude that the sextic $g''+tf^2$ is psd
for sufficiently large $t>0$.
\end{proof}

\begin{cor}
Let $T\subset\P^2(\R)$ be any set with $|T|=8$, having no 4 points on
a line, no 7 on a conic. The cone $P_6(T)$ in $I_6(2T)$ has full
dimension four, while $\Sigma_6(T)=P_6(T)\cap I_3(T)^2$ has dimension
three.
\qed
\end{cor}

\begin{proof}
Follows from Theorem \ref{8ptsexpsdnonsos} and Proposition
\ref{rappel}(c).
\end{proof}

\begin{rem}
For most sets $T\subset\P^2(\R)$ of 8 points (no 4 on a line, no 7 on
a conic) we have seen that there exists a non-sos psd sextic $q$ that
vanishes on $T$ plus an extra point outside~$T$. It is clear that
such $q$ exists if and only if $T$ has the following property: There
exists a nonsingular cubic $X$ through $T$ for which the (unique)
point $Q_X\in X(\R)$ with $[Q_X+\sum_{P\in T}P]=\tau_X$ satisfies
$Q_X\notin T$. We do not know if there exists a set $T$ for which
this property fails.
\end{rem}

\begin{rems}\Label{rationalsexticsrems}%
\hfil\smallskip

1.\
The psd sextic $g$ is said to be \emph{exposed} if
any $g'\in P_6$ with $V_\R(g)\subset V_\R(g')$ satisfies $g'\in
\R_\plus g$. Clearly such $g$ spans an extreme ray of $P_6$. The
converse is not true in general, however Straszewicz's theorem (e.g.\
\cite{ro} Theorem 18.6)
shows that every extreme form is the limit of a sequence of exposed
forms.
It is easy to see that a form $g\in P_6\setminus\Sigma_6$ is exposed
if and only if $|V_\R(g)|=10$. Indeed, when $|V_\R(g)|\le9$ there is
a cubic $p$ through $V_\R(g)$, so $p^2$ is another psd sextic
vanishing on $V_\R(g)$. Conversely, if $|V_\R(g)|=10$, and if $g'$ is
another psd sextic with $V_\R(g)\subset V_\R(g')$, then $g,\,g'$ have
intersection number $\ge10\cdot2^2=40$, implying $g'\in\R_\plus g$
since $g$ is irreducible (see \cite{rz07} Theorem 7.2 for this
argument).
\smallskip

2.\
Since an absolutely irreducible plane sextic can have at most 10
singular points, the preceding remark shows that any exposed and
non-sos psd sextic defines a rational curve. Therefore, and by
Straszewicz's theorem, every extreme form in $P_6\setminus\Sigma_6$
defines a rational sextic. This argument was used in \cite{bhors} to
show that the Zariski closure in $|\scrO_{\P^2}(6)|$ of the extreme
curves in $P_6\setminus\Sigma_6$ is contained in the Severi variety
$S_{6,0}$ of plane rational sextics.
In fact, the authors proved (loc.\,cit.\ Theorem~2) that the Zariski
closure is equal to $S_{6,0}$, by observing that $S_{6,0}$ is irreducible
of dimension~17, and by producing a 17-dimensional local family of
extreme forms in $P_{3,6}\setminus\Sigma_{3,6}$, based on an analysis
from \cite{rz07}.
See also Remark \ref{dimrem} below.
\smallskip

3.\
Proposition \ref{10thzero} shows that for a generically chosen
admissible set $S$, the extreme sextic $q_S$ has a tenth real zero,
and is therefore exposed. Conversely, it follows from Corollary
\ref{zsfgpsdsex} that every exposed sextic $q$ is covered by
construction \ref{sumcomp}. Actually this happens in ten different
ways, since there are ten ways of choosing nine points out of ten.
\end{rems}

\begin{rem}\Label{dimrem}%
Corollary \ref{admiffexpsdnonsossext} gives a very explicit argument
for the fact that the set of extreme forms in $P_{3,6}\setminus
\Sigma_{3,6}$ has (projective) dimension~$17$.
Indeed, fixing a plane cubic curve $X$ with at most nodes as
singularities, there is an 8-parameter family of admissible subsets
$S$ of $X_\reg(\R)$. Adding the 9 parameters for choosing $X$, this
together shows that the admissible sets form a $17$-dimensional
family. The map assigning to each admissible $S$ the extreme psd
sextic $q_S$ is an explicit parametrization of the extreme sextics
in $P_6\setminus\Sigma_6$ with at least $9$ real zeros. Generically,
this parametrization is a $10:1$-map.
\end{rem}

We finally discuss a geometric way to find the 10th zero of the
extreme psd sextic $q_S$, for $S\subset\P^2(\R)$ admissible. The
following construction is extracted from Coble's remarkable paper
\cite{cob}. We work over $\C$ in the next proposition.

\begin{prop}\Label{nonic}%
\emph{(Coble)}
Let $T\subset\P^2$ be a set with $|T|=8$, no $4$ points on a line
and no $7$ on a conic. Choose cubic forms $f,\,f'$ with $I_3(T)=
\spn(f,f')$ and a sextic form $g$ with $I_6(2T)=\spn(f^2,\,ff',\,
f'^2,\,g)$ (c.f.\ \ref{rappel}). Let $M$ be the 9th point of
intersection of $f$ and $f'$, and let $j=\det J(f,f',g)$, the
Jacobian determinant.
\begin{itemize}
\item[(a)]
The curve $N_T=V(j)$ has degree~$9$ and depends only on $T$.
\item[(b)]
Every point of $T$ is a triple point of $N_T$.
\item[(c)]
Assume that the cubic $X=V(f)$ is nonsingular. Then for $Q\in X$,
$Q\notin T\cup\{M\}$ we have
$$\dim I_6(2T+2Q)>1\ \iff\ 2Q\sim2M\text{ on }X\ \iff\ Q\in N_T.$$
Moreover $M\notin N_T$ except when $M\in T$.
\end{itemize}
\end{prop}

Note that for $X=V(f)$ nonsingular, (c) gives $27$ intersection
points of $X$ and $N_T$, at least in the generic case $M\notin T$,
namely the three points $Q\ne M$ with $2Q\sim2M$, together with the
eight points in $T$ where the intersection index is~$3$.

\begin{proof}
(a) follows from elementary properties of Jacobians, and (b) is
readily seen in local coordinates.
For (c) assume that $X=V(f)$ is nonsingular, and let $Q\in X$,
$Q\notin T\cup\{M\}$. When $h$ is a sextic that is singular in
$T\cup\{Q\}$ and not a multiple of $f^2$, then $\gcd(f,h)=1$ and
$\div_X(h)=2Q+2\sum_{P\in T}P=2(Q-M)+\div_X(f')$, hence $2Q\sim2M$
on $X$. On the other hand, $2Q\sim2M$ implies the existence of a
sextic $h\in I_6(2T)$ with $\gcd(f,h)=1$ and $i_Q(f,h)\ge2$,
hence $J(f,f',h)$ is singular at $Q$.
Conversely let $j(Q)=0$. Since $f'(Q)\ne0$, this means that
$i_Q(f,g)\ge2$, and arguing as in \ref{singularize0} we find a cubic
$p$ such that $g+pf$ is singular in $T\cup\{Q\}$, showing
$\dim I_6(2T+2Q)>1$. When $M\notin T$ then $f$ and $f'$ meet
transversally at $M$, so $j(M)\ne0$.
\end{proof}

\begin{cor}\Label{9thzerononic}%
Let $S\subset\P^2(\R)$ be an admissible set of nine points, let
$P\in S$, and put $T=S\setminus\{P\}$. Then $P$ lies on the nonic
curve $N_T$ (\ref{nonic}).
\end{cor}

\begin{proof}
There exists a sextic $g\in I_6(2T)\setminus I_3(T)^2$ that is
singular in $P$, for example the extreme psd sextic $q_S$
(\ref{nontatgs}). Since $N_T$ has equation $j=\det J(f,f',g)=0$ where
$I_3(T)=\spn(f,f')$, it is clear that $j(P)=0$.
\end{proof}

\begin{rem}\Label{remfind10thzero}%
Let $S\subset\P^2(\R)$ be an admissible set, and assume that the
extreme psd sextic $q_S$ has a 10th real zero $Q\notin S$.
(Generically this is the case, see Proposition \ref{10thzero}.)
Corollary \ref{9thzerononic} offers a way of finding $Q$ (and
therefore $q_S$), at least up to a finite choice. Indeed, $Q$ lies
on the intersection $N_{S\setminus\{P\}}\cap N_{S\setminus\{P'\}}$,
for any choice of two points $P\ne P'$ in~$S$.
See Example \ref{beispielsymm} for an illustration.
\end{rem}

\begin{rem}\Label{char10ptsets}%
One may wonder about a characterization of the 10-point real zero
sets of psd sextics, or in other words, of the real zero sets of
exposed psd sextics. Such a characterization has been asked for in
\cite{clr80}. Let $U\subset\P^2(\R)$ be a set with $|U|=10$. If a
psd sextic $q$ with $V_\R(q)=U$ exists then
every 9-point subset of $U$ is admissible. Conversely assume that
$U\setminus\{P\}$ is admissible for every point $P\in U$. Does it
follow that there exists a psd sextic $q$ with $V_\R(q)=U$? This
seems likely, but we haven't been able to prove it. At any rate, our
results permit to check, for any given $U$, whether there is a psd
sextic with zeros in $U$ (and to find it in the positive case): Pick
$P\in U$, decide whether $S=U\setminus\{P\}$ is admissible, and if
yes, proceed to determine the extreme sextic $q_S$. If there exists a
psd sextic through $U$ it has to be $q_S$, so a necessary and
sufficient condition is that $q_S(P)=0$.
\end{rem}

%-------------------------------------------------------------------%

\section{Examples}

We are going to illustrate the general construction of admissible
sets in the plane and their associated extreme psd sextics, as well
as several phenomena discussed in the previous sections. For
typographical simplicity we sometimes use homogeneous coordinates
$(x,y,z)$ instead of $(x_0,x_1,x_2)$.

\begin{lab}\Label{bsp1}%
For a first example we work in affine coordinates $(x,y)$. Let
$X=V(f)$ be the elliptic curve with affine equation
$$f(x,y,1)\>=\>y^2-x^3-x^2-4.$$
Let $+$ denote the abelian group law on $X(\R)$ with neutral element
the point at infinity. The group $X(\R)$ has one connected component,
with real $2$-torsion point $Q=(-2,0)$. Let $P=(0,2)$, a non-torsion
point on $X$. The nine points $S=\{Q,\ \pm P,\ \pm2P,\ Q\pm P,\
Q\pm2P\}$ on $X$ add up to $Q$ in the group law of $X$. Therefore the
set $S$ is admissible.
Each of the four triples $(P,\,Q+P,\,Q-2P)$, $(-P,\,Q-P,\,Q+2P)$,
$(Q,2P,\,Q-2P)$, $(Q,-2P,Q+2P)$ lies on a line, and the six points
$(\pm P,\,\pm2P,\, Q\pm P)$ lie on a conic.
The product is a sextic that is singular in the points of $S$, namely
$q=l_1l_2l_3l_4p$ where $l_1=3x+y-2$, $l_2=3x-y-2$, $l_3=2x+y+4$,
$l_4=2x-y+4$ and $p=y^2-2x^2-2x-4$.
By construction, $q$ has constant (nonnegative) sign on $X(\R)$.
The minimal value $t\in\R$ for which the sextic $q_t=q+tf^2$ is psd
gives the extreme sextic $q_S=q_t$ in $P_6(S)\setminus\Sigma_6(S)$.
One finds that $q_S$ is exposed with 10th real zero at $(\alpha,0)$,
where $\alpha\approx4.25925$ is the real root of $5x^3-15x^2-24x-12$.
The corresponding value for $t$ is $t=\frac1{49}(165\alpha^2+
60\alpha+1156)\approx89.89509$.
\end{lab}

\begin{lab}\Label{bsp3}%
For a second example let $X=V(x_0x_1x_2)$, union of three lines in
general position.
Write $P_0(a)=(0:a:1)$, $P_1(a)=(1:0:a)$ and $P_2(a)=(a:1:0)$ for
$a\in\R^*$. The Picard group of $X$ is an extension $1\to\R^*\to
\Pic(X)\to\Z^3\to0$. Given $a_i,\,a'_i\in\R^*$ ($i=0,1,2$) we have
$$\sum_{i=0}^2P_i(a_i)\>\sim\>\sum_{i=0}^2P_i(a'_i)\text{ \ in }
\Div(X)\quad\iff\quad a_0a_1a_2\>=\>a'_0a'_1a'_2.$$
So if $a_0,\,a_1,\,a_2\in\R^*$, the three points $P_i(a_i)$
($i=0,1,2$) lie on a line if and only if $a_0a_1a_2=-1$.
Therefore, $S\subset X_\reg(\R)$ is admissible if and only if there
exist $a_i,\,b_i,\,c_i\in\R^*$ with $|\{a_i,b_i,c_i\}|=3$ ($i=0,1,2$)
such that
$$S\>=\>\Bigl\{P_i(a_i),\>P_i(b_i),\>P_i(c_i)\colon i=0,1,2\Bigr\}$$
and
$$a_0a_1a_2\cdot b_0b_1b_2\cdot c_0c_1c_2\>=\>1.$$
For an explicit example let
$$S\>=\>\Bigl\{P_i(\pm1)\colon i=0,1,2\Bigr\}\cup
\Bigl\{P_0\bigl(\frac32\bigr),\>P_1\bigl(\frac13\bigr),\>
P_2(-2)\Bigr\}.$$
A sextic that is singular in the points of $S$ is
$$q\>=\>(x_0+x_1+x_2)(x_0+x_1-x_2)(x_0-x_1+x_2)(x_0-x_1-x_2)
(x_0+2x_1-3x_2)^2$$
The extreme psd sextic associated to $S$ is $q_S=q+t(x_0x_1x_2)^2$
where $t\approx114.68148$.
Again, $q_S$ is exposed with 10th real zero $(\alpha:\beta:1)$ where
$\alpha\approx-0.64185$ is the smallest real root of
$27x^3-111x^2-59x+15$ and $\beta=\frac9{16}(1-\alpha^2)+2\alpha
\approx-0.95295$. (The above number $t$ is also a rational expression
in $\alpha$.)
\end{lab}

\begin{lab}\Label{bsp6}%
This is an example of an admissible set $S$ for which the extreme psd
sextic $q_S$ has only the 9 real zeros in $S$. Let $S$ consist of the
9 points
$$(0:1:\pm1),\ (1:0:\pm1),\ (1:\pm1:0),\ (3:0:2),\ (3:2:0),\
(5:4:4).$$
One can show that the set $S$ is admissible, the unique cubic
$X=V(f)$ through $S$ being given by
$$f\>=\>8\,(x+y+z)(2x^2+3y^2+3z^2-5xy-5xz-6yz)+195\,xyz.$$
$X$ is an elliptic curve for which $X(\R)$ has two connected
components.
Let $l=2x-3y-3z$, and let $l_1,\,l_2,\,l_3,\,l_4$ be the four linear
forms $x\pm y\pm z$. The extreme psd sextic with zeros in $S$ is
$$q_S\>=\>2496\cdot l_1l_2l_3l_4l^2+(888\,f+273\,xyz)^2$$
The sextic $q_S$ has no real zero beyond $S$, rather the point
$(0:1:-1)\in S$ is an $A_3$-singularity of $q_S$.
\end{lab}

\begin{lab}\Label{bsprob}%
The Robinson sextic \cite{rob} is the unique psd sextic through the
10 points that arise from permuting the coordinates of $(1:1:\pm1)$
and $(0:1:\pm1)$. It has symmetric equation
$$R\>=\>(x^6+y^6+z^6)-(x^4y^2+x^4z^2+y^4x^2+y^4z^2+z^4x^2+z^4y^2)
+3x^2y^2z^2.$$
Let $T$ consist of the 8 points $(\pm1:\pm1:1)$, $(\pm1:0:1)$,
$(0:\pm1:1)$. Then $I_3(T)$ is the pencil generated by $f=x^3-xz^2$
and $f'=y^3-yz^2$. Therefore $I_6(2T)$ is spanned by $f^2$, $ff'$,
$f'^2$ and $R$.
The nonic $N_T=\det J(f,f',R)$ has equation
$$j\>=\>z(z^2-x^2)(z^2-y^2)(x^2+xy+y^2-z^2)(x^2-xy+y^2-z^2).$$
Except for the line $z=0$, every irreducible component of $N_T$ is
contained in some cubic from the pencil $I_3(T)$.
Therefore, every admissible set $S=T\cup\{P\}$ has its 9th point $P$
on the line $z=0$.

Conversely fix $u\in\R\cup\{\infty\}$, let $P_u=(1:u:0)$ (with
$P_\infty=(0:1:0)$) and put $S_u=T\cup\{P_u\}$. Then the set $S_u$ is
admissible for all $u\notin\{0,\infty\}$. Indeed, there is a unique
cubic through $S_u$, namely $X_u=V(f_u)$ with $f_u=u^3f-f'$.
And the definiteness condition is satisfied since $X_u(\R)$ is
connected (resp.\ $X_u$ is a union of a line and a conic with two
real intersection points for $u=\pm1$). For $u=0,\,\infty$ the set
$S_u$ fails to be admissible since $P_u$ is a triple point of $X_u$.

Let $u\ne0,\,\infty$. To find the extreme sextic $q_u:=q_{S_u}$, and
possibly its 10th real zero $Q_u$, we may use the Coble nonic as in
\ref{remfind10thzero}. So we may pick $P'=(1:1:-1)$ (or any other
point in $T$), put $T'_u=S_u\setminus\{P'\}$, compute the nonic
$N_{T'_u}$ and intersect it with $N_T$. One finds that $N_{T'_u}$
intersects the line $z=0$ of $N_T$ in the four real points $(1:0:0)$,
$(0:1:0)$, $(u:1:0)$ and $(1:-u:0)$, besides $P_u$.
The sextics in $I_6(2S_u)$ through the first three points are
indefinite. Therefore we must have $Q_u=(1:-u:0)$. The corresponding
sextic is
$$q_u\>=\>u^6(u^2+2)(x^3-xz^2)^2-3u^4(x^2-z^2)(y^2-z^2)(x^2+y^2-z^2)
+(2u^2+1)(y^3-yz^2)^2$$
which must be the exposed psd sextic with ten zeros in $T\cup
\{(1:\pm u:0)\}$. For $u=\pm1$ we get $q_u=3R$.

Note that the example of $T$
is special since there exist lines containing three points of $T$
(and conics containing six points of $T$). These lines and conics are
necessarily irreducible components of the nonic $N_T$.
\end{lab}

\begin{lab}\Label{beispielsymm}%
For every real number $u$, the symmetric sextic
$$q_u\>=\>a\sum x_i^4x_j^2+b\sum x_i^4x_jx_k+c\sum x_i^3x_j^3+
d\sum x_i^3x_j^2x_k+e(x_0x_1x_2)^2$$
with
$$a=u^2,\ b=-2(u^2-2),\ c=2u^2,\ d=-2(u^2+u+2),\ e=6(u^2+4u+2)$$
is singular in the points that arise from
$$(1:0:0),\ (1:-1:0),\ (1:1:u),\ (1:1:1)$$
by permuting the coordinates. For $u\ne1$ these are 10 points.
The following hold:
\begin{itemize}
\item[(a)]
$q_u$ is psd and not sos for $u<-2$ or $u>1$,
\item[(b)]
$q_u$ is a sum of two squares for $u=-2$, and a square for $u=1$,
\item[(c)]
$q_u$ is indefinite for $-2<u<1$.
\end{itemize}
Here is a way to see this. Let $S_u$ be the set of points obtained by
permuting the coordinates of $(1:0:0)$, $(1:-1:0)$ and $(1:1:u)$
(so $|S_u|=9$ for $u\ne1$). We need to find the values of $u$ for
which $S_u$ is admissible. For $u=1$ we have $|S_u|=7$, for $u=-2$
the line $x_0+x_1+x_2=0$ contains six points of $S_u$. For any $u$
the symmetric cubic $X_u=V(f_u)$ with
$$f_u\>=\>2(u^2+u+1)\,x_0x_1x_2-u\sum x_i^2x_j$$
passes through $S_u$, and for $u\notin\{1,-2\}$ it is the only cubic
through $S_u$. For the singular values $u\in\{-2,-\frac12\}$, $X_u$
is the union of a line and a conic without real intersection point,
and for $u\in\{-1,0\}$, $X_u$ is a triangle whose singularities lie
in $S_u$. Therefore $S_u$ is not admissible for
$u\in\{-2,-1,-\frac12,0,1\}$.
For all $u\notin\{-2,\,-1,\,-\frac12,\,0,\,1\}$, the cubic $X_u$ is
nonsingular, and $X(\R)$ has two connected components.
For any of these values $u$, the divisor class $[\scrO_{X_u}(3)]
-\sum_{P\in S_u}[P]$ on $X_u$ is a nonzero $2$-torsion element in
$\Pic(X_u)$. To find the values $u$ for which this class is definite
(and hence $S_u$ is admissible), we follow the geometric approach in
\ref{admissgeom}.
Let $P=(1:0:0)\in S_u$ and put $T_u=S_u\setminus\{P\}$. The 9th
intersection point of the pencil $I_3(T_u)$ is $M_u=(u:1:1)$ (which
happens to lie in $T_u$).
The tangents to $X_u$ at $P$ and $M_u$ are $t_P=x_1+x_2$ and
$t_{M_u}=x_0-ux_1-ux_2$, so they intersect at $(0:1:-1)\in T_u$.
A~geometric sketch shows that the product $t_{M_u}\cdot t_P$ is
indefinite on $X_u(\R)$ if and only $-2<u<1$.

The discussion shows that for $-2<u<1$, the only psd sextic through
$S_u$ is $f_u^2$.
On the other hand, there exist psd non-sos sextics through
$S_u$ when $u<-2$ or $u>1$, and the unique extreme one is $q_u$.
\end{lab}

%===================================================================%

\end{document}